\documentclass[12pt]{amsart}
\usepackage{setspace}
\usepackage{amssymb,latexsym,amsmath,amsbsy,amsthm,amsxtra,amsgen,amscd}

\newcommand {\Z} {\mathbb{Z}}

\renewcommand {\P} {\mathbb{P}}

\newtheorem{Theo}{Theorem}

\newtheorem{Prop}[Theo]{Proposition}

\begin{document}

\author[Holroyd]{Alexander E.\ Holroyd}

\address[Alexander E. Holroyd]{Microsoft Research}

\author[Liggett]{Thomas M.\ Liggett}

\address[Thomas M. Liggett]{University of California, Los Angeles}

\title{Symmetric $1-$Dependent Colorings of the Integers}
\date{July 15, 2014}
\maketitle

\begin{abstract}
\sloppypar In a recent paper by the same authors, we constructed a stationary
$1-$dependent $4-$coloring of the integers that is invariant under
permutations of the colors.  This was the first stationary $k-$dependent
$q-$coloring for any $k$ and $q$. When the analogous construction is carried
out for $q>4$ colors, the resulting process is not $k-$dependent for any $k$.
We construct here a process that is symmetric in the colors and $1-$dependent
for every $q\geq 4$.  The construction uses a recursion involving Chebyshev
polynomials evaluated at $\sqrt{q}/2$.
\end{abstract}

\section{Introduction}
By a (proper) $q-$coloring of the integers, we mean a sequence $(X_i:i\in
\Z)$ of $[q]-$valued random variables satisfying  $X_i\neq X_{i+1}$ for all
$i$ (where $[q]:=\{1,\ldots,q\}$). The coloring is said to be stationary if
the (joint) distribution of $(X_i:i\in \Z)$ agrees with that of
$(X_{i+1}:i\in \Z)$, and $k-$dependent if the families $(X_i:i\leq m)$ and
$(X_i:i>m+k)$ are independent of each other for each $m$. In \cite{HL}, we
gave a construction of a stationary $1-$dependent $4-$coloring of the
integers that is invariant under permutations of the colors. When the same
construction is carried out for $q>4$ colors, the resulting distribution is
not $k-$dependent for any $k$. Of course, the $1-$dependent $4-$coloring is
also a $1-$dependent $q-$coloring for every $q>4$, and one may obtain other
$1-$dependent $q-$colorings by splitting a color into further colors using an
independent source of randomness.  However, these colorings are not symmetric
in the colors. We give here a modification of the process of \cite{HL} that
is symmetric in the colors and $1-$dependent for every $q\geq 4$. Here is our
main result.

\begin{Theo} \label{maintheorem}For each integer $q\geq 4$,
there exists a stationary  $1-$dependent $q-$coloring of the integers that is
invariant in law under permutations of the colors and under the reflection
$(X_i:i\in\Z)\mapsto (X_{-i}:i\in\Z)$.
\end{Theo}

Our construction is given in the next section. Sections 3 and 4 provide some
preliminary results and the proof of Theorem \ref{maintheorem} respectively.

\section{The construction}

For $x=(x_1,x_2,\dots, x_n)\in[q]^n$, we will write $P(x)=\P(X_1=x_1,\dots,
X_n=x_n).$ To motivate the construction, we begin by noting that the
finite-dimensional distributions $P$ of the $4-$coloring in \cite{HL} are
defined recursively by $P(\emptyset)=1$ and
\begin{equation}\label{4Pdef}P(x)=\frac 1{2(n+1)}\sum_{i=1}^nP(\widehat x_i)\end{equation}
for proper $x\in [4]^n$, where $\widehat x_i$ is obtained from $x$ by
deleting the $i$th entry in $x$. Of course, even if $x$ is proper, $\widehat
x_i$ may not be. So the definition is completed by setting $P(x)=0$ for $x$'s
that are not proper.

For general $q\geq 4$, we will now allow the coefficients in the defining sum to depend on $i$ as well as $n$. Considering many special cases, and the constraints imposed by the $1-$dependence requirement, we were led to define
\begin{equation}\label{Pdef}P(x)=\frac1{D(n+1)}\sum_{i=1}^nC(n-2i+1)P(\widehat x_i)\end{equation}
for proper $x\in [4]^n$, in terms of two sequences $C$ and $D$. Again
motivated by computations in special cases, we take
\begin{align*}
C(n)&=T_n\bigl(\sqrt q/2\bigr), & n\geq 0; \\
D(n)&=\sqrt q\,U_{n-1}\bigl(\sqrt q/2\bigr), & n\geq 1,
\end{align*}
 where $T_n$ and $U_n$ are the Chebyshev polynomials of the first
and second kind respectively.

There are several standard equivalent definitions of Chebyshev polynomials.
One is
\begin{equation}\label{cheb}
T_n(u)=\cosh(nt) \quad\text{and}\quad U_n(u)=\frac{\sinh[(n+1)t]}{\sinh(t)},
\quad\text{where } u=\cosh(t).
\end{equation}
A variant definition using trigonometric functions (e.g.\ (22:3:3-4) of
\cite{OMS}) is easily seen to be equivalent by taking $t$ imaginary; the
hyperbolic function version is convenient for arguments $u\geq 1$.  Another
definition is
$$T_n(u)=\sum_{k=0}^{\lfloor\frac{n}2\rfloor}\binom{n}{2k}u^{n-2k}(u^2-1)^k
\quad\text{and}\quad
U_n(u)=\sum_{k=0}^{\lfloor\frac{n}2\rfloor}\binom{n+1}{2k+1}u^{n-2k}(u^2-1)^k.$$
That this is equivalent to \eqref{cheb} follows from e.g.\ (22:3:1-2) of
\cite{OMS}.

If $x$ is not a proper coloring, we take $P(x)=0$ as before.  We extend both
sequences $C$ and $D$ to all integer arguments by taking $C(n)$ and $D(n)$ to
be even and odd functions of $n$ respectively (in accordance with
\eqref{cheb}).

Observe that $C(n)$ and $D(n)$ are strictly positive for $q\geq 4$ and $n\geq
1$, and therefore $P(x)$ is strictly positive for all proper $x$.  Note also
that $C(n-2i+1)/D(n+1)$ is rational; therefore so is $P(x)$. (The factors of
$\sqrt q$ cancel).  When $q=4$ we have $C(n)=1$ and $D(n)=2n$, and so
(\ref{Pdef}) reduces to (\ref{4Pdef}) in this case. As we will see, the fact
that the coefficients in (\ref{Pdef}) depend on $i$ substantially complicates
the verifications of the required properties of $P$.

Here are a few examples of cylinder probabilities generated by (\ref{Pdef}).
\begin{gather*}P(1)=\frac 1q,\quad P(12)=\frac1{q(q-1)},
\quad P(121)=\frac 1{q^2(q-1)},\quad P(123)=\frac 1{q^2(q-2)},\\
P(1212)=\frac{q-3}{q^2(q-1)(q^2-3q+1)},\quad P(1234)=\frac
1{q^2(q^2-3q+1)}.\end{gather*}

\section{Preliminary results}

Chebyshev polynomials satisfy a number of standard identities. They lead to
identities satisfied by the sequences $C$ and $D$. The first three in the
proposition below are examples of this. The fourth is a consequence of the
third one. Before stating them, we record some values of $C$ and $D$ to
facilitate checking computations here and later.
\begin{gather*}
C(0)=1,\quad C(1)=\frac{\sqrt q}2,\quad C(2)=\frac{q-2}2,\quad
C(3)=\frac{\sqrt q(q-3)}2,\quad C(4)=\frac{q^2-4q+2}2.\\D(0)=0,\quad
D(1)=\sqrt q,\quad D(2)=q,\quad D(3)=\sqrt q(q-1),\quad D(4)=q(q-2).
\end{gather*}

\begin{Prop}\label{firstprop}  For $j,k,\ell,m,n\in \Z$, the following identities hold.
\begin{gather}
\label{firstpropa}2C(m)C(n)=C(m+n)+C(n-m).\\
\label{firstpropc}\frac{q-4}{2q}D(m)D(n)=C(m+n)-C(n-m).\\
\label{firstpropb}2C(m)D(n)=D(m+n)+D(n-m).\\
\label{firstpropd}C(j+k)D(k+\ell)=C(k)D(j+k+\ell)-C(\ell)D(j).
\end{gather}
\end{Prop}

\begin{proof} The first three parts are immediate consequences of (22:5:5-7)
in \cite{OMS}, or 22.7.24-26 in \cite {AS}, if
$m$ and $n$ are nonnegative.  None of the identities is changed by changing
the sign of either $m$ or $n$.  Therefore, they hold for all $m$ and $n$.
Alternatively, the identities may be checked directly from \eqref{cheb} using
the product formulae for hyperbolic functions. For (\ref{firstpropd}),
replace the products of $C$'s and $D$'s by sums of $D$'s using
(\ref{firstpropb}), and then use the fact that $D$ is an odd function.
\end{proof}

Next we verify some identities that involve both the sequences $C$ and $D$
and the measure $P$ defined by \eqref{Pdef}. For the statement of the second
part of the next result,  let
$$Q(x)=\frac1{D(n+1)}\sum_{i=1}^nC(2i)P(\widehat x_i)\quad\text{and}\quad
 Q^*(x)=\frac1{D(n+1)}\sum_{i=1}^nC(2n-2i+2)P(\widehat x_i)$$
for $x\in [q]^n$. The first part of Proposition \ref{secondprop} is needed in
proving the second part, which plays a key role in the proof of consistency
and $1-$dependence of $P$. Note the similarity between the left side of
(\ref{Dsum}) and the right side of (\ref{Pdef}).

\begin{Prop}\label{secondprop} If $n\geq 1$, and $x$ is a proper coloring of length $n$, then
\begin{gather}\label{Dsum}\sum_{i=1}^nD(n-2i+1)P(\widehat x_i)=0;\\
\label{Qis}Q(x)=Q^*(x)= P(x)C(n+1).
\end{gather}
\end{Prop}

\begin{proof}
For the first statement, let $R$ be the set of proper colorings, and
$\widehat x_A$ be obtained by deleting the entries $x_i$ for $i \in A$ from
$x$. The proof of (\ref{Dsum}) is by induction on $n$, the length of $x$. The
identity is easily seen to be true if $n\leq 2$. Suppose that (\ref{Dsum}) is
true for all $x$ of length $n-1$, and let $x\in R$ have length $n$.  For
those $i$ with $\widehat x_i\in R$, applying
 (\ref{Dsum}) gives
\begin{equation}\label{sum4}\sum_{j=1}^{i-1}D(n-2j)P(\widehat x_{i,j})+
\sum_{j=i+1}^nD(n-2j+2)P(\widehat x_{i,j})=0.\end{equation} On the other hand, if $\widehat
x_i\notin R$, then $1<i<n$ and
\begin{equation}\label{sum5}P(\widehat x_{i,j})=0\text{ if }|j-i|>1
\text{ and }P(\widehat x_{i-1,i})=P(\widehat x_{i,i+1}).\end{equation} The
left side of (\ref{Dsum}) for $x$ can be written, using the definition of
$P(\widehat x_i)$ and then (\ref{firstpropb}), as
\begin{equation}\label{proof2}\begin{aligned}&=
\frac1{D(n)}\sum_{\substack{1\leq i\leq n:\\ \widehat x_i\in R}}D(n-2i+1)
\bigg[\sum_{1\leq j<i}C(n-2j)P(\widehat x_{i,j})+\sum_{i<j\leq n}C(n-2j+2)P(\widehat x_{i,j})\bigg]\\
&=\frac1{2D(n)}\sum_{\substack{1\leq j<i\leq n:\\ \widehat x_i\in R}}
[D(2n-2i-2j+1)+D(2j-2i+1)]P(\widehat x_{i,j})\\
&\quad+\frac1{2D(n)}
\sum_{\substack{1\leq i<j\leq n:\\ \widehat x_i\in R}}[D(2n-2i-2j+3)+D(2j-2i-1)]
P(\widehat x_{i,j}).\end{aligned}\end{equation}
Rearranging, and ignoring the $2D(n)$ in the denominator, gives
\begin{equation}\begin{aligned}\label{sum6}\sum_{i=1}^n{\bf 1}
[\widehat x_{i}\in R]\bigg[&\sum_{j=1}^{i-1}[D(2n-2i-2j+1)+D(2j-2i+1)]
P(\widehat x_{i,j})\\+&
\sum_{j=i+1}^n[D(2n-2i-2j+3)+D(2j-2i-1)]P(\widehat x_{i,j})\bigg].\end{aligned}\end{equation}
We must show that (\ref{sum4}) and (\ref{sum5}) imply that (\ref{sum6}) is zero.

We would like to write (\ref{sum6}) as a linear combination of expressions that vanish because of
 (\ref{sum4}) and (\ref{sum5}) as follows.
\begin{equation}\label{sum7}\sum_{\substack{1\leq i\leq n:\\ \widehat x_i\in R}}
\alpha_i\bigg[\sum_{j=1}^{i-1}D(n-2j)P(\widehat x_{i,j})
+\sum_{j=i+1}^nD(n-2j+2)P(\widehat x_{i,j})\bigg]
+\sum_{\substack{1\leq i\leq n:\\ \widehat x_i\in R}}\sum_{j=1}^n\beta_{i,j}
P(\widehat x_{i,j}),\end{equation}
where $\beta_{i,i}=\beta_{i,i-1}+\beta_{i,i+1}=0$.
If $1\leq i< j\leq n$, the coefficient of $P(\widehat x_{i,j})$ in (\ref{sum6}) is
\begin{equation}\label{sum9}\begin{aligned}
&{\bf 1}[\widehat x_j\in R]\,\bigl [D(2n-2i-2j+1)+D(2i-2j+1)\bigr]\\
{}+&{\bf 1}[\widehat x_i\in R]\,\bigl[D(2n-2i-2j+3)+D(2j-2i-1)\bigr].
\end{aligned}\end{equation}
The coefficient of $P(\widehat x_{i,j})$ in (\ref{sum7}) is
\begin{equation}\label{sum10}{\bf 1}[\widehat x_j\in R]\alpha_jD(n-2i)
+{\bf 1}[\widehat x_i\in R]\alpha_iD(n-2j+2)+{\bf 1}[\widehat x_i\notin
R]\beta_{i,j} +{\bf 1}[\widehat x_j\notin R]\beta_{j,i}.\end{equation} We
need to choose the $\alpha$'s and $\beta$'s so that (\ref{sum9}) and
(\ref{sum10}) agree. If $\widehat x_i,\widehat x_j\in R$, this says
$$D(2n-2i-2j+1)+D(2n-2i-2j+3)=\alpha_jD(n-2i)+\alpha_i D(n-2j+2)$$
since $D$ is an odd function. It may sound unreasonable to expect to solve
this system, since there are $n$ unknowns and $\binom n2$ equations. However,
$D$ satisfies relations that make this possible. Solving the equations for
small $n$ suggests trying $\alpha_i=2C(n-2i+1)$. The fact that this choice
solves these equations for all choices of $n,i,j$ then follows from
(\ref{firstpropb}) and the fact that $D$ is odd. If $\widehat x_i\notin R$
and $\widehat x_j\notin R$, (\ref{sum9}) and (\ref{sum10}) agree if
$\beta_{i,j}+\beta_{j,i}=0$. If $\widehat x_i\in R$ and $\widehat x_j\notin
R$, they agree if
$$D(2n-2i-2j+3)+D(2j-2i-1)=\alpha_iD(n-2j+2)+\beta_{j,i}.$$
Using (\ref{firstpropb}) again gives $\beta_{j,i}=2D(2j-2i-1)$. Similarly, if
$\widehat x_i\notin R$ and $\widehat x_j\in R$, they agree if
$\beta_{i,j}=2D(2i-2j+1)$. With these choices, $\beta$ is anti-symmetric, and
$\beta_{k,k-1}=2D(1)$ and $\beta_{k,k+1}=2D(-1)$, so
$\beta_{k,k-1}+\beta_{k,k+1}=0$ as required. This completes the induction
argument.

For (\ref{Qis}), consider the case of $Q$ first. Use the definition of $P$ to
write the right side of (\ref{Qis}) as
$$\frac{C(n+1)}{D(n+1)}\sum_{i=1}^nC(n-2i+1)P(\widehat x_i).$$
Using (\ref{firstpropa}), this becomes
$$\frac1{2D(n+1)}\sum_{i=1}^n C(2n-2i+2)P(\widehat x_i)+\frac12Q(x).$$
Therefore, we need to prove that
$$\sum_{i=1}^n\bigl[C(2n-2i+2)-C(2i)\bigr]\,P(\widehat x_i)=0.$$
But by (\ref{firstpropc}), this follows from (\ref{Dsum}). The proof for
$Q^*$ is similar.
\end{proof}

\section{Proof of the main result}
We will often write $x_1x_2\cdots x_n$ instead of $(x_1,x_2,\dots, x_n)$
below.  If $x\in[q]^m$ and $y\in[q]^n$, let $xy$ denote  the word $x_1\cdots
x_my_1\cdots y_n\in [q]^{m+n}$
\begin{proof}[Proof of Theorem~\ref{maintheorem}]
We first need to show that the finite dimensional distributions defined in
(\ref{Pdef}) are consistent, i.e.,  that
\begin{equation}\label{consistency}\sum_{a\in[q]}P(xa)=P(x),
\qquad x\in[q]^n,\;n\geq 0.\end{equation} This is true if $x$ is not proper,
since then $xa$ is also not proper, and so both sides vanish. For proper $x$,
the proof is by induction on $n$. Note that for $a\in[q]$,
$$P(a)=\frac{C(0)}{D(2)}=\frac 1q,$$
so $\sum_{a\in[q]}P(a)=1$. This gives (\ref{consistency}) for $n=0.$ Suppose
it holds for all $x\in[q]^{n-1}$ with $n\geq 1$. Then for proper $x\in[q]^n$,
using the induction hypothesis in the second equality,
$$\begin{aligned}\sum_{a\in[q]}P(xa)&=
\sum_{a\neq x_n}\frac1{D(n+2)}\bigg[\sum_{i=1}^nC(n-2i+2)P(\widehat x_ia)+C(-n)P(x)\bigg]\\
&=\frac1{D(n+2)}\bigg[\sum_{i=1}^nC(n-2i+2)P(\widehat
x_i)-C(-n+2)P(x)+(q-1)C(-n)P(x)\bigg].\end{aligned}$$ The middle term in the
second line accounts for the missing term $a=x_n$ when the inductive
hypothesis is applied to the case $i=n$ (since $\widehat x_n x_n=x$).
Using $(j,k,\ell)=(1,n-2i+1,2i)$ in (\ref{firstpropd}) gives
$$\frac{C(n-2i+2)}{D(n+2)}=\frac{C(n-2i+1)}{D(n+1)}-\frac{C(2i)D(1)}{D(n+2)D(n+1)}.$$
Therefore
$$\sum_{a\in[q]}P(xa)=P(x)-\frac{Q(x)}{D(n+2)}-
\frac{C(n-2)}{D(n+2)}P(x)+(q-1)\frac{C(n)}{D(n+2)}P(x).$$
This is $P(x)$, as required, by (\ref{Qis}) and the fact that
$$(q-1)C(n)=C(n-2)+C(n+1)D(1),$$
which is obtained by taking $(j,k,\ell)=(2,-n,n+1)$ in (\ref{firstpropd}),
and then canceling a factor of $\sqrt q$.

Invariance of the measure under permutations of colors and translations is
immediate from the definition.  Invariance under reflection amounts to
checking $P(x)=P(x_n\cdots x_1)$, which follows from the fact that the
coefficients of $\widehat x_i$ and $\widehat x_{n-i+1}$ in (\ref{Pdef}),
which are $C(n-2i+1)$ and $C(-n+2i-1)$ respectively, are equal by the
symmetry of $C$.

For $1-$dependence, we need to show that for $x\in[q]^m$ and $y\in[q]^n$ with
$m,n\geq 0$,
$$P(x*y)=P(x)P(y),$$
where the * means that there is no constraint at the single site between $x$
and $y$. This is again true if $x$ or $y$ is not proper since then both sides
are zero. For proper $x$ and $y$, the proof is by induction, but now on
$m+n$. The statement is immediate if $m=0$ or $n=0$. So, we take $m\geq 1$
and $n\geq 1$.

There are two cases, according to whether or not $xy$ is a proper coloring,
i.e., whether $x_m$ and $y_1$ are equal or different. Assume first that
$x_m=y_1$. Without loss of generality, take their common value to be 1. Then
using the definition of $P$, including the fact that $P(xy)=0$,
\begin{multline} P(x*y)=\sum_{a\in[q]}P(xay)
=\frac1{D(n+m+2)}\sum_{a\neq 1}\bigg[\sum_{i=1}^mC(n+m-2i+2)P(\widehat x_iay)\\
\shoveright{\qquad+C(n-m)P(xy)+\sum_{j=1}^nC(n-m-2j)P(xa\widehat y_j)\bigg]}\\
=\frac1{D(n+m+2)}\bigg[\sum_{i=1}^mC(n+m-2i+2)P(\widehat x_i*y)+
\sum_{j=1}^nC(n-m-2j)P(x*\widehat y_j)\bigg].
\end{multline}
Using the induction hypothesis, this becomes
$$P(x*y)=\frac1{D(n+m+2)}\bigg[P(y)\sum_{i=1}^mC(n+m-2i+2)
P(\widehat x_i)+P(x)\sum_{j=1}^nC(n-m-2j)P(\widehat y_j)\bigg].$$
Taking $(j,k,l)=(n,m-2i+1,i)$ in (\ref{firstpropd}) gives
$$\frac{C(n+m-2i+2)}{D(n+m+2)}=\frac{C(m-2i+1)}{D(m+1)}-\frac{C(2i)D(n+1)}{D(m+1)D(n+m+2)}.$$
Similarly,
$$\frac{C(m+2j-n)}{D(n+m+2)}=\frac{C(2j-n-1)}{D(n+1)}-\frac{C(2n-2j+2)D(m+1)}{D(n+1)D(n+m+2)}.$$
Therefore, since $C(\cdot)$ is even,
$$P(x*y)=P(y)\bigg[P(x)-\frac{D(n+1)}{D(n+m+2)}Q(x)\bigg]\\
+P(x)\bigg[P(y)-\frac{D(m+1)}{D(n+m+2)}Q^*(y)\bigg].
$$
By (\ref{Qis}),
$$P(x*y)=P(x)P(y)\bigg[2-\frac{C(m+1)D(n+1)+C(n+1)D(m+1)}{D(n+m+2)}\bigg].$$
Taking $(j,k,l)=(n,m-2i+1,i)$ in (\ref{firstpropd}),
we see that the expression in brackets above is 1, as required.

Assume now that $x_m\neq y_1$, say $x_m=1$ and $y_1=2$. Then
\begin{multline}
P(x*y)=\sum_{a\in[q]}P(xay)
=\frac1{D(n+m+2)}\sum_{a\neq 1,2}\bigg[\sum_{i=1}^mC(n+m-2i+2)P(\widehat x_iay)\\
\shoveright{\qquad+C(n-m)P(xy)+\sum_{j=1}^nC(n-m-2j)P(xa\widehat y_j)\bigg]}\\
\frac1{D(n+m+2)}\bigg[\sum_{i=1}^mC(n+m-2i+2)P(\widehat
x_i*y)+\sum_{j=1}^nC(n-m-2j)P(x*\widehat y_j)\bigg]
\end{multline}
as in the previous case. However, in the previous case, the term $P(xy)$
dropped out because $xy$ was not a proper coloring. In this case, the term
$(q-2)C(n-m)P(xy)$ is cancelled by the terms $-P(xy)C(n-m+2)$ and
$-P(xy)C(n-m-2)$, which arise from
$$\sum_{a\neq 1,2}P(\widehat x_may) =P(\widehat x_m*y)-P(xy)
\text{ and }\sum_{a\neq 1,2}P(xa\widehat y_1)=P(x*\widehat y_1)-P(xy).$$
The fact that the overall coefficient of $P(xy)$ vanishes is a consequence
(\ref{firstpropa}) with $m=2$, since $2C(2)=q-2$. The rest of the proof is
the same as in the case $x_m=y_1$ above.
\end{proof}

\bibliographystyle{amsalpha}

\end{document}